\documentclass{amsart}
\usepackage{amsmath}
\usepackage{amsthm}
\usepackage{amssymb}
\usepackage{xcolor}
\usepackage[shortlabels]{enumitem}
\usepackage{mathtools}
\usepackage{bm}
\usepackage{tikz-cd}
\usepackage[normalem]{ulem}
\usetikzlibrary{matrix,arrows}

\newcommand{\Z}{\mathbb{Z}}

\newcommand{\bp}{\begin{pmatrix}}
\newcommand{\ep}{\end{pmatrix}}
\newtheorem{theorem}{Theorem}[section]
\newtheorem{lemma}{Lemma}

\theoremstyle{definition}

\newtheorem{rmk}{Remark}

\theoremstyle{remark}

\title{Finite Gelfand Pairs and Cracking Points of the Symmetric Groups}
\author{Faith Pearson}
\author{Anna Romanov}
\author{Dylan Soller}

\date{August 2019}

\numberwithin{equation}{section}

\begin{document}

\begin{abstract}
    Let $\Gamma$ be a finite group. Consider the wreath product $G_n := \Gamma^n \rtimes S_n$ and the subgroup $K_n := \Delta_n \times S_n\subseteq G_n$, where $S_n$ is the symmetric group and $\Delta_n$ is the diagonal subgroup of $\Gamma^n$. For certain values of $n$ (which depend on the group $\Gamma$), the pair $(G_n, K_n)$ is a Gelfand pair. It is not known for all finite groups which values of $n$ result in Gelfand pairs. Building off the work of Benson--Ratcliff \cite{BR}, we obtain a result which simplifies the computation of multiplicities of irreducible representations in certain tensor product representations, then apply this result to show that for $\Gamma = S_k, \ k \geq 5$, $(G_n,K_n)$ is a Gelfand pair exactly when $n = 1,2$.
\end{abstract}

\maketitle

\section{Introduction}

 Let $G$ be a finite group and $K$ a subgroup of $G$. Denote by $L(G)$ the set of complex-valued functions on $G$. This is an algebra under the convolution product

$$
f \star g (x) = \frac{1}{|G|} \sum_{y \in G} f(xy^{-1})g(y).
$$

\noindent The pair $(G,K)$ is said to be a {\em Gelfand pair} if the subalgebra $L(K \backslash G / K)$ of $K$-biinvariant functions in $L(G)$ is commutative.

 Gelfand pairs are well-studied in the context of Lie groups, where there is an analogous definition in terms of the algebra of integrable $K$-biinvariant functions on the group $G$. (See, for example, \cite{BR2}.) In the Lie group setting, the Gelfand pair structure can  be used to construct irreducible unitary representations of $G$ from representations of the subgroup $K$. Historically, these techniques played a pivotal role in describing the representation theory of semi-simple Lie groups \cite{H}. In the finite group setting, the theory of Gelfand pairs is less-developed, and has found surprising applications outside of group theory including statistics, experimental design, and combinatorics. For example, in \cite{D}, Diaconis uses finite Gelfand pairs to determine the rate at which certain Markov chains converge to stationary distributions, and the authors of \cite{B} apply finite Gelfand pairs to the study of association schemes. In \cite{AC}, finite Gelfand pairs are used to study parking functions, a useful tool in algebraic combinatorics. 

 This paper concerns a construction introduced by Aker--Can in \cite{AC} which produces families of finite Gelfand pairs associated to a fixed finite group. The construction proceeds as follows. Given a finite group $\Gamma$, the symmetric group $S_n$ acts on $\Gamma^n$ by permuting the factors, and we form the wreath product $G_n := \Gamma^n \rtimes S_n$ of $\Gamma$ with $S_n$. Let $\Delta_n$ be the diagonal subgroup of $\Gamma^n$. Then $K_n := \Delta_n \times S_n$ is a subgroup of $G_n$, and for certain values of $n$, the pair $(G_n,K_n)$ is a Gelfand pair.

In particular, when $\Gamma$ is abelian, $(G_n,K_n)$ is a Gelfand pair for all values of $n$ \cite{CST}. Such Gelfand pairs are relevant in the study of parking functions when $\Gamma$ is cyclic \cite{AC}. For non-abelian $\Gamma$, Benson--Ratcliff establish the following two results.

\begin{enumerate}
\item \cite[Theorem 1.2]{BR} The pair $(G_{|\Gamma|},K_{|\Gamma|})$ is not a Gelfand pair.
\item \cite[Theorem 1.1]{BR} There is some integer $N(\Gamma)$ with $3 \leq N(\Gamma) \leq |\Gamma|$ such that $(G_n,K_n)$ is a Gelfand pair for $n < N(\Gamma)$, and is not a Gelfand pair for $n \geq N(\Gamma)$. 
\end{enumerate}

We refer to $N(\Gamma)$ as the {\em cracking point} of $\Gamma$ and say that $\Gamma$ {\em cracks} at $N(\Gamma)$.

Aker--Can showed through GAP computations that there are groups for which this upper bound is reached and also groups for which this lower bound is reached \cite{AC}. For example, the symmetric group $S_3$ has a cracking point of $6$, whereas the group $GL(2,\mathbb{F}_3)$ has a cracking point of $3$. On the other hand, Benson--Ratcliff show that in certain infinite families of groups with no bound on order, the cracking point remains constant. For example, they show that for all odd primes $p$, the dihedral group $D_p$ has a cracking point of $6$ \cite{BR}. In general, the relationship between the finite group $\Gamma$ and its cracking point remains rather mysterious. The main result of this paper is to establish the cracking points of the symmetric groups.
\begin{theorem}
\label{thm}
Let $G_n:=(S_k)^n\rtimes S_n$ and $K_n:=\Delta_n \times S_n$, where $\Delta_n \subset (S_k)^n$ is the diagonal subgroup. For $k\geq 5$, the pair $(G_n, K_n)$ is a Gelfand pair for $n=1,2$ and is not a Gelfand pair for $n\geq 3$; that is, in the notation above, $N(S_k)=3$ for $k\geq 5$. Moreover, $N(S_4) = 4$ and $N(S_3) = 6$. 
\end{theorem}
We prove Theorem \ref{thm} using a general observation (Lemma \ref{Me Prop}) which simplifies the computation of cracking points.

Our paper is structured as follows. In Section 2, we discuss a decomposition of the $G_n$-representation $L(G_n/K_n)$, following the setup in \cite{BR}. This gives us the vocabulary necessary to establish our key observation, Lemma \ref{Me Prop}. In Section 3, we apply Lemma \ref{Me Prop} to prove our main result. 

\subsection*{Acknowledgments}

We would like to extend our gratitude toward Gail Ratcliff for introducing us to this problem and for sharing her expertise and enthusiasm about Gelfand pairs. We would also like to thank the University of Utah Department of Mathematics REU program for funding our project.


\section{Background}
Let $\Gamma$ be a finite group and $K_n \subset G_n$ as above. By general results about Gelfand pairs, the pair $(G_n,K_n)$ is a Gelfand pair if and only if the left quasi-regular representation $ind_{K_n}^{G_n}(triv_{K_n})$ of $G_n$ in $L(G_n/K_n)$ is multiplicity free \cite[Ch. 3F Thm. 9]{D}. Benson--Ratcliff give a decomposition of the space $L(G_n/K_n)$ into irreducible $G_n$-representations in \cite{BR}. In this section, we review some of the details of this decomposition in order to establish our key lemma.  

As $G_n = \Gamma^n \rtimes S_n$, it is perhaps unsurprising that the irreducible representations of $G_n$ can be constructed from those of $\Gamma$ and certain subgroups of $S_n$. The construction is as follows. Let $\{\pi_\ell\}_{\ell \in S}$ be the irreducible representations of $\Gamma$, where $S$ is an indexing set in bijection with the conjugacy classes of $\Gamma$. The irreducible representations of $\Gamma^n$ are all of the form $\pi := \pi_{\ell_1} \hat{\otimes} \cdots \hat{\otimes} \pi_{\ell_{n-1}} \hat{\otimes} \pi_{\ell_n}$, where $\hat{\otimes}$ denotes the exterior tensor product, and $\ell_i \in S$ (note that we allow for $\ell_i = \ell_k$ for $i \neq k$). The symmetric group $S_n$ acts on any such $\pi$ by permuting the factors, and we denote by $S_\pi$ the stabilizer of $\pi$ in $S_n$. Denote by $\omega$ the intertwining representation of $S_\pi$; that is,
\[
\omega: S_\pi \rightarrow GL(V_1 \otimes \cdots \otimes V_n),\ \  \omega(\sigma)(v_1 \otimes \cdots \otimes v_n) = v_{\sigma^{-1}(1)} \otimes \cdots \otimes v_{\sigma^{-1}(n)} 
\]
where $V_i$ is the vector space of the representation $\pi_{\ell_i}$. Then for any irreducible representation $\rho$ of $S_\pi$, the induced representation $R_{\pi,\rho} := ind_{\Gamma^n \rtimes S_\pi}^{G_n}((\pi \circ \omega)\hat{\otimes} \rho)$ is an irreducible $G_n$-representation, and all irreducible representations of $G_n$ are of this form \cite[Sec. 3.2]{BR}. Throughout this paper, for a representation $\pi$ of a group $G$, we denote by $\chi_\pi$ its character. 

Benson--Ratcliff provide a useful method for determining the multiplicity of $R_{\pi,\rho}$ in $L(G_n/K_n)$. In particular, they show that the dimension of the space of $K_n$-fixed vectors in $R_{\pi,\rho}$ (which is equal to the multiplicity of $R_{\pi,\rho}$ in $L(G_n/K_n)$) is equal to the dimension of the space of $K_\pi$-fixed vectors in $(\pi \circ \omega)\hat{\otimes}\rho$, where $K_\pi := \Delta_n \rtimes S_\pi$ \cite[Lem. 3.3]{BR}. This can be calculated by taking the inner product of the character of $(\pi \circ \omega)\hat{\otimes}\rho$, with the trivial character on $K_\pi$:
\begin{equation}
\label{trivial char}
\frac{1}{|\Delta_n \times S_\pi|}\sum_{(\delta,\sigma) \in K_\pi}\chi_{\pi \circ \omega}(\delta,\sigma)\chi_\rho(\sigma) = \frac{1}{|S_\pi|}\sum_{\sigma \in S_\pi} \bigg( \frac{1}{|\Delta_n|}\sum_{\delta \in \Delta_n} \chi_{\pi \circ \omega}(\delta,\sigma) \bigg) \chi_\rho(\sigma). 
\end{equation}
The inner sum on the right hand side of (\ref{trivial char}) is a class function on $S_\pi$. This class function plays an important role in our story so we give it a name:
\begin{equation}
\label{M pi}
M_\pi(\sigma) := \frac{1}{|\Delta_n|}\sum_{\delta\in \Delta_n} \chi_{\pi \circ \omega}(\delta,\sigma).
\end{equation}
Equation (\ref{trivial char}) determines the coefficient of $\chi_\rho$ in the decomposition of $M_\pi$ into irreducible characters of $S_\pi$. Therefore, we see that $(G_n,K_n)$ is a Gelfand pair if and only if for each choice of $\pi$, the coefficient of $\chi_\rho$ in $M_\pi$ is less than or equal to 1 for all irreducible representations $\rho$ of $S_\pi$.  

Now we wish to highlight a key observation concerning $M_\pi(e)$, the value of $M_\pi$ on the identity $e \in S_\pi$. First, note that for $\delta \in \Delta_n$,
\[
 \chi_{\pi \circ \omega}(\delta,e) = \prod_{i = 1}^{n}\chi_{\pi_{\ell_i}}(\delta).
\]
Thus, substituting this into (\ref{M pi}), we have   
\begin{equation}
\label{M pi identity}
M_\pi(e) = \frac{1}{|\Delta_n|} \sum_{\delta\in \Delta_n} \big(\prod_{i = 1}^{n}\chi_{\pi_{\ell_i}}(\delta)\big) = \frac{1}{|\Gamma|}\sum_{C} \big(\prod_{i = 1}^{n}\chi_{\pi_{\ell_i}}(C)\big)|C|,  
\end{equation}
where $C$ runs over the conjugacy classes of $\Gamma$. The second equality follows from the fact that $\Delta_n \simeq \Gamma$ by the obvious isomorphism. Now, the right hand side of (\ref{M pi identity}) is also equal to the inner product of the $\Gamma$-representation $\pi_{\ell_1} \otimes \cdots \otimes \pi_{\ell_{n-1}}$ with $\pi_{\ell_n}$ whenever $\chi_{\pi_{\ell_n}}$ is real-valued. Hence we have proven the following result.

\begin{lemma}
\label{Me Prop}
Let $\Gamma$ be a finite group, and let $\{\pi_\ell\}_{\ell \in S}$ be the irreducible representations of $\Gamma$. Then for $\pi = \pi_{\ell_1} \hat{\otimes}\cdots \hat{\otimes}  \pi_{\ell_n}$, $M_\pi(e)$ is equal to the multiplicity of $\pi_{\ell_n}$ in $\pi_{\ell_1} \otimes\cdots \otimes  \pi_{\ell_{n-1}}$, if $\chi_{\pi_{\ell_n}}$ is real-valued.
\end{lemma}

\begin{rmk}
Note that the product in (\ref{M pi identity}) is not changed by reordering the $\chi_{\pi_{\ell_i}}$. Thus, more generally, we have shown that $M_\pi(e)$ is equal to the multiplicity of $\pi_{\ell_i}$ in $\pi_{\ell_1} \otimes \cdots \otimes \pi_{\ell_{i-1}} \otimes \pi_{\ell_{i + 1}} \otimes \cdots \otimes \pi_{\ell_n}$ whenever $\chi_{\pi_{\ell_i}}$ is real-valued. For our purposes, we will only consider the case when $i = n$, as in Lemma \ref{Me Prop}. 
\end{rmk}

 With this, we can simplify the calculations used in computing cracking points. In particular, we can use Lemma \ref{Me Prop} to make statements about $M_\pi$ based solely on the dimensions of $\pi_{\ell_n}$ and $\pi_{\ell_1} \otimes\cdots \otimes  \pi_{\ell_{n-1}}$, which allows us to circumvent the necessity for complete character tables in some cases. An example of such utility is given in the proof of Theorem \ref{thm} below. 

\section{Cracking Points of $S_k$}

\noindent In this section we use Lemma \ref{Me Prop} to compute the cracking points of the symmetric groups. We start with the following observation.

\begin{lemma}
\label{SummerLemma}
Let $S_\pi$ be the stabilizer of $\pi = \pi_{\ell_1} \hat{\otimes} \cdots \hat{\otimes} \pi_{\ell_n}$ in $S_n$. If $M_\pi(e) > \sum_{\rho \in \widehat{S}_\pi} \dim{\rho}$, then $(G_n,K_n)$ is not a Gelfand pair. 
\end{lemma}

\begin{proof} Because $M_{\pi}$ is a class function on $S_{\pi}$, it can be expressed uniquely as a linear combination of irreducible characters $\chi_{\rho}$ of $S_{\pi}$: 
\[
M_{\pi}= \sum_{\rho \in \widehat{S}_\pi} a_\rho \chi_{\rho}
\]
\noindent where $\{a_\rho\}_{\rho \in \widehat{S}_\pi}$ are complex coefficients. Now by (\ref{trivial char}), $a_\rho= \langle M_{\pi}, \chi_{\rho} \rangle$ counts the dimension of the space of $K_{\pi}$-fixed vectors in the $\Gamma^n \rtimes S_\pi$-representation $(\pi \circ \omega) \hat{\otimes} \rho$. Thus, we see that $a_\rho \in \Z^{\geq 0}$ for all $\rho \in \widehat{S}_\pi$. Therefore, if 
\[
M_{\pi}(e) = \sum_{\rho \in \widehat{S}_\pi}a_\rho \dim{\rho} > \sum_{\rho \in \widehat{S}_\pi} \dim{\rho}
\]
there must be some $\rho \in \widehat{S}_\pi$ such that $a_\rho > 1$. Hence, $(G_n, K_n)$ is not a Gelfand pair. 
\end{proof}

\subsection*{Proof of Theorem 1}

 Fix $k\geq 5$, and let $\pi_m$ be the highest dimensional irreducible representation of $S_k$. We claim that there is an irreducible representation $\psi$ of $S_k$ such that for $\pi = \pi_m \hat{\otimes} \pi_m \hat{\otimes} \psi$, there is some irreducible character $\chi_\rho$ of $S_\pi$ which has a coefficient greater than 1 in the decomposition of $M_\pi$. By \cite[Lem. 3.3]{BR}, this implies that the irreducible $G_3$-representation $R_{\pi, \rho}$ has multiplicity greater than $1$ in $L(G_3/K_3)$, and hence $(G_3,K_3)$ is not a Gelfand pair.

 To show that such a representation $\psi$ exists, there are two cases to consider. The first case is when $\psi = \pi_m$ and $S_\pi = S_3$. The other case is when $\psi \neq \pi_m$, in which case $S_\pi = S_2 \times S_1 \simeq S_2$. We will prove that $(G_3,K_3)$ is not a Gelfand pair by showing that $M_\pi(e) > 4$ in the first case, $M_\pi(e) > 2$ in the second case, and then applying Lemma 3.1. By Lemma \ref{Me Prop}, this is equivalent to showing that the coefficient of $\pi_m$ in $\pi_m \otimes \pi_m$ is greater than 4, or that the coefficient of $\pi_i$ in $\pi_m \otimes \pi_m$ is greater than 2 for some $\pi_i \in \widehat{S}_k$ different than $\pi_m$. To do this, we will show that the following inequality holds for all $k \geq 5$:
\[
(\dim{\pi_{m}})^2 > 4\dim{\pi_{m}} + \sum_{\pi_i \in \widehat{S}_k,\ \pi_i \neq \pi_m} 2\dim{\pi_{i}}.
\]
As $\pi_m$ is of maximal dimension in $\widehat{S}_k$, it is enough to show that 
\[
(\dim{\pi_{m}})^2 \geq 4\dim{\pi_{m}} +  2(p(k) - 1)\dim{\pi_{m}}
\]
where $p(k)$ is the number of partitions of $k$, which is equal to the number of irreducible representations of $S_k$. Simplifying, this amounts to establishing the inequality
\begin{equation}
\label{ineq}
\dim{\pi_m} \geq  2p(k) + 2.
\end{equation}

An asymptotic lower bound is given for $\dim{\pi_m}$ in \cite[Thm.1]{VK}, namely 
\[
\dim{\pi_m} \geq e^{-c\sqrt{k}}\sqrt{k!}
\]
where $c = \frac{\pi}{\sqrt{6}}$. Similarly, an asymptotic upper bound was found for $p(k)$ in \cite{HR}:
\[
p(k) \leq \frac{1}{4k\sqrt{3}} e^{\pi\sqrt{\frac{2k}{3}}}.
\]
Combining these results with (\ref{ineq}), we see that $(G_3,K_3)$ fails to be a Gelfand pair if
\[
e^{-c\sqrt{k}}\sqrt{k!} \geq \frac{1}{2k\sqrt{3}} e^{\pi\sqrt{\frac{2k}{3}}} + 2.
\]
This is equivalent to the condition that the ratio
\[
r(k) := \frac{2k\sqrt{3}e^{-c\sqrt{k}}\sqrt{k!} - 4k\sqrt{3}}{e^{\pi\sqrt{\frac{2k}{3}}}}
\]
is greater than or equal to 1. A direct calculation shows that this holds for $k = 12$. For $k > 12$, note that, after replacing $k!$ with the Gamma function $\Gamma(k+1)$ restricted to the positive real axis, the derivative $\frac{d}{dk}r(k)$ is positive and hence $r(k)$ is increasing.  Thus, $r(k) \geq 1$ for $k \geq 12$ and $N(S_k) = 3$ in that case. The rest we calculate through a case-by-case analysis.


 The cracking points of $S_k$ for $k = 4,5,6,$ and $7$ can be computed directly using their character tables. Here we show $N(S_5) = 3$ as an example. To do this, we will calculate $M_\pi$ directly for a specific choice of an irreducible representation $\pi$ of $S_5 \times S_5 \times S_5$. Consider the following partial character table of $S_5$, which contains the characters of the highest dimensional and second highest dimensional irreducible representations:
\[
\begin{matrix}
 &(1) & (10) & (15) & (20) & (20) & (24) & (30) \\
 & I & 2 & 2,2 & 3 & 3,2 & 5 & 4 \\
 \pi_1 & 5 & 1 & 1 & -1 & 1 & 0 & 1 \\
 \pi_2 & 6 & 0 & -2 & 0 & 0 & 1 & 0
\end{matrix}
\]

Now let $\pi = \pi_2 \hat{\otimes} \pi_2 \hat{\otimes} \pi_1$, which has a stabilizer of $S_\pi \simeq S_2$ in $S_3$. Then calculating $M_\pi(\sigma)$ directly from (\ref{M pi identity}), we see that $M_\pi(\sigma) = 2$ for both $\sigma$ in $\widehat{S}_\pi$. Hence, $M_\pi = 2\chi_{triv}$, and $(G_3,K_3)$ is not a Gelfand pair for $\Gamma = S_5$. Similarly, for $k = 6$ and $k = 7$ we take $\pi$ to be the representation consisting of two copies of the highest dimensional irreducible representation of $S_k$ and one copy of the second highest. Again from (\ref{M pi identity}), we calculate that $M_\pi(e) = 4$ in the case of $S_6$, and $M_\pi(e) = 5$ for $S_7$. Thus, by Lemma \ref{SummerLemma}, $(G_3,K_3)$ is not a Gelfand pair in either case. Finally, for $\Gamma = S_4$, taking $\pi$ to be four copies of the standard representation of $S_4$ suffices to show that $(G_4,K_4)$ is not a Gelfand pair. Calculating the decomposition of $M_\pi$ into irreducible $S_\pi$-representations for each choice of $\pi$ when $n = 3$, one finds no cases of multiplicity, and hence $N(S_4) = 4$.   

For $S_8$ through $S_{11}$, we show directly that (\ref{ineq}) holds. The computations are contained in the table below. The values for $\dim{\pi_{m}}$ are given in \cite[Ch. 7B]{D}. 

\begin{center}
 \begin{tabular}{||c c c||} 
 \hline
 $k$ & $\dim{\pi_{m}}$ & $2p(k) + 2$ \\ [0.5ex] 
 \hline\hline
 8 & 90 & 46 \\
 9 & 216 & 62 \\
 10 & 768 & 86 \\
 11 & 2310 & 114 \\
 \hline
\end{tabular}
\end{center}

\qed

\bibliographystyle{alpha}
\bibliography{bibliography}

\end{document}